\documentclass[12pt,epsfig,tikz,multi]{amsart}
\usepackage[utf8]{inputenc}
\usepackage[T1]{fontenc}
\usepackage{latexsym}
\usepackage{breqn}
\usepackage{amsrefs}
\usepackage{float}
\restylefloat{table}
\usepackage{amsmath,amscd,amssymb}
\usepackage{caption}
\usepackage{subcaption}
\usepackage{enumitem}
\usepackage{diagbox}
\include{plaquettemacro}
\usepackage{hyperref}
\usepackage{lineno}
\usepackage{color}

 2

\newtheorem{theorem}{Theorem}[section]
\newtheorem{prop}{Proposition}[section]

\newtheorem{definition}[theorem]{Definition}
\newtheorem{example}[theorem]{Example}

\newtheorem{remark}[theorem]{Remark}

\numberwithin{equation}{section}
\newtheorem*{theorem*}{Theorem}

\title{APPROXIMATING SINGULARITIES BY A CUSPIDAL-EDGE ON A MAXFACE}
\author{PRADIP KUMAR 
}
\address{Department of Mathematics, Shiv Nadar University, Dadri 201314, Uttar pradesh, India.\\
Email Address: pradip.kumar@snu.edu.in}
\author{SAI RASMI RANJAN MOHANTY}
\address{Department of Mathematics, Shiv Nadar University, Dadri 201314, Uttar pradesh, India.\\
Email Address: sm743@snu.edu.in}
\date{}
\begin{document}
\maketitle
\begin{abstract}
    We give necessary and sufficient conditions on the singular Bj\"{o}rling data to the singular Bj\"{o}rling problem’s solution has a prescribed nature of singularity. As an application, for a given maxface with a particular type of singularity, we find a sequence of maxfaces, all having a cuspidal edge.
\end{abstract}

\section{Introduction}
Maximal immersions are zero mean curvature immersions in the Lorentz-Minkowski space $\mathbb E_1^3$. These are very similar to the minimal surface in $\mathbb R^3$, but if we allow some singularities (where maps are not immersions), the theory of these two diﬀers. Maximal surfaces with singularity are called generalized maximal surfaces. Singularity on the generalized maximal surface is branched and non-branched. Non-branched singular points are those points where limiting tangent space does not collapse, and it contains a light-like vector. Various aspects of non branched singularities have been studied in \cite{Estudillo1992}, \cite{Fujimori2007}, \cite{imaizumi2008}, \cite{Kim2007}, \cite{KOBAYASHI1984}, \cite{OT2018}, \cite{ UMEHARA2006} etc.

Umehara and Yamada in \cite{UMEHARA2006}, proved that every non branched maximal immersions as a map in $\mathbb R^3$ turn out to be frontal. They called non branched maximal immersions as maxface and discussed when this becomes the front near a singular point. Cuspidal-edge, swallowtails, cuspidal crosscaps, cuspidal butterﬂies, and cuspidal $S_1^-$ are few singularities that appear on a maxface $X$ as front or frontal.

On cuspidal-edges of the front, Saji, Umehara, and Yamada \cite{Saji2009}
  introduced the singular curvature function  which is closely related to the behavior of the Gaussian curvature of a surface near cuspidal edges.  Further, Martins and  Saji  \cite{MARTINS2018209}, study differential geometric properties of cuspidal edges with boundary and given  several differential geometric invariants.   Toshizumi Fukui \cite{FukuiCuspidaledge} also studied the  local differential geometry of cuspidal edges.

Not all maxfaces are front  but if we approximate singularities on maxfaces with a cuspidal-edge  (that is the first kind of singularities of fronts)  then it may help us to understand the existence of some invariant related to other types of singularities as we have  for the cuspidal-edge (\cite{FukuiCuspidaledge}, \cite{MARTINS2018209},\cite{Saji2009}).   This idea motivates us to  approximate singularities by  a cuspidal-edge.  

In this article in section 4, we construct a  sequence (there may be many others) of maxfaces with a cuspidal edge that ``converges" to other singularities  like shrinking or folded or as in the table \ref{intro:tab}.    To prove this,  we shall give necessary and suﬃcient conditions on the singular Björling data $\{\gamma,L\}$ such that it has a cuspidal-edge at some point.

Along with the cuspidal edge, in this article, we ﬁnd the necessary and suﬃcient conditions on the singular Björling data $\{\gamma,L\}$ such that its corresponding maxface has swallowtails, cuspidal crosscaps, cuspidal butterﬂies, and cuspidal $S_1^-$. We summarize the conditions (given in the propositions: \ref{proposition_cuspidaledge}, \ref{proposition_swallowtails}, and \ref{propsition_crosscaps} ) here in the table \ref{intro:tab}.

\begin{table}[h]
\small
\begin{tabular}{|c |c| c| c| c| c| c|c|c| }\hline
\diagbox[width=4cm]{{Nature of}\\ {singularities at $p$}}{{Function's value}\\ at $p$}&$\gamma^\prime$&$L$&$\gamma^{\prime\prime}$&$\gamma^{\prime\prime\prime}$&$L^\prime$&$L^{\prime\prime}$&$\gamma_1^{\prime}\gamma_2^{\prime\prime}-\gamma_1^{\prime\prime}\gamma_2^{\prime}$&$L_1L_2^{\prime}-L_1^{\prime}L_2$\\
\hline
Cuspidal-edge&$\neq0$&$\neq0$&--&--&--&--&$\neq 0$&$\neq 0$\\
\hline
Swallowtails&$=0$&$\neq0$&$\neq0$&--&--&--&--&$\neq 0$\\
\hline
Cuspidal butterflies&$=0$&$\neq 0$&$=0$&$\neq 0$&--&--&--&$\neq 0$\\
\hline
Cuspidal $S_1^-$&$\neq0$&$=0$&--&--&$=0$&$\neq 0$&$\neq 0$&--\\
\hline
Cuspidal-Crosscaps&$\neq0$&$=0$&--&--&$\neq 0$&--&$\neq 0$&--\\
\hline
\end{tabular}
\caption{}
\label{intro:tab}
\end{table}

We believe the above table is very useful and apart from finding required convergent sequence, it may be a starting point of studying suitable interpolation problem (finding a maxface containing two  disjoint curves with prescribed nature of singularities along the curve). In \cite{LOPEZ20072178}, L\'{o}pez  discussed a kind of interpolation problem where he proves the existence of maximal immersion (not the generalized)  spanning two disjoint circular
contours. Some discussion about finding maxface with two interpolating singular curve can be found in \cite{RPR2016}. But we believe, a general discussion requires an initial setup like the conditions as in the table \ref{intro:tab}. In article \cite{brander_2011}, David Brander has discussed similar conditions for the case of the non-maximal CMC surfaces with the special data. 

The discussion of this article is close to \cite{brander_2011}, \cite{Kim2007}, and \cite{UMEHARA2006}. 

\section{Preliminary}
This section reviews the deﬁnition of maxface, Weierstrass-Enneper representation, and the singular Björling problem. 

The Lorentz-Minkowski space $\mathbb E_1^3$ is a vector space $\mathbb R^3$ with metric $\langle, \rangle : \mathbb R^3\times\mathbb R^3 \to \mathbb R$ deﬁned by $\langle(a_1,b_1,c_1),(a_2,b_2,c_2)\rangle := a_1a_2 + b_1b_2 -c_1c_2$ and the generalised maximal immersion is an immersion of a Riemann surface (with boundary) $M$ to $\mathbb E_1^3$, such that pullback metric on $M$ does not vanish identically and it is positive definite wherever metric does't vanish. Moreover at non degenerate points mean curvature is zero.  Maxfaces are those generalized maximal immersions where singularities are only those points of $M$ where the limiting tangent plane has a light-like vector. We have the following representation of the maxface.

\subsection{Weierstrass-Enneper representation \cite{UMEHARA2006}} For a maxface $X : M \to\mathbb E_1^3$, there is a pair $(g,\omega)$ of meromorphic function and a holomorphic 1-form on $M$ such that  $|g|$ is not identically equal 1 and for  $\Phi := (1 + g^2,i(1-g^2),-2g)\omega$, the map $X$ is given by $X(p) := Re \int_{0}^{p}\Phi.$ 

For a maxface, with the help of Weierstrass data $(g,\omega)$, below we define functions $\alpha, \beta,$ and $\eta$ as in  \cite{Fujimori2007}, \cite{OT2018}, and \cite{UMEHARA2006}. 
\begin{definition}
\label{alpha_beta_gamma}
At  $p\in M,$  let  $(U,z)$ be a coordinate chart and $\omega= f\,dz$, we define 
   $$ \alpha(z)=\frac{g^{\prime}(z)}{g^2(z)f(z)},\; \beta(z)=\frac{g(z)}{g^{\prime}(z)}\alpha^{\prime}(z),\,\, { \rm and }\,\, \eta(z)=\frac{g(z)}{g^{\prime}(z)}\beta^{\prime}(z). $$
\end{definition}
These functions help us to check the nature of singularity on a maxface. In \cite{Fujimori2007}, \cite{OT2018}, and \cite{UMEHARA2006}, we find  criterion to check the nature of singularity. We mention it here. 

\begin{table}[h]
\begin{tabular}{|c c c c c|}
\hline
  Re$(\alpha)\neq 0$ & Im$(\alpha)\neq 0$ & & &\hspace{-1cm}$\Leftrightarrow p$ is a cuspidal-edge\\\hline
  Re$(\alpha)\neq 0$ &Im$(\alpha)= 0$ & Re $(\beta)\neq 0$& &\hspace{-1cm}$\Leftrightarrow p$ is a swallowtails\\\hline
Re$(\alpha)\neq 0$ & Im$(\alpha)= 0$ & Re $(\beta)= 0$ &Im$(\eta)\neq 0$&$\Leftrightarrow p$ is a cuspidal butterlflies\\\hline
Re$(\alpha)=0$ & Im$(\alpha)\neq 0$ & Im$(\beta)= 0$ &Re$(\eta)\neq 0$&$\Leftrightarrow p$ is a cuspidal $S_1^-$\\\hline
Re$(\alpha)= 0$ & Im$(\alpha)\neq  0$ & Im $(\beta)\neq 0$& &$\Leftrightarrow p$ is a cuspidal crosscaps\\\hline
\end{tabular}
\caption{}
\end{table}
\subsection{Singular Björling problem \cite{Kim2007}} We explain the singular Bj\"{o}rling problem in the following.
\begin{definition}(Singular Björling data \cite{Kim2007}).
\label{bjorling_data}
Let $\gamma : I\to\mathbb E_1^3$ be a real analytic null curve and $L : I\to\mathbb E_1^3$ be a real analytic null vector ﬁeld such that for all $u \in I$, $\gamma^\prime (u)$ and $L(u)$ are proportional,  and $\gamma^{\prime}(u)$ and $L(u)$ do not vanish simultaneously.  Such $\{\gamma,L\}$ is said to be a singular Bj\"{o}rling data.
\end{definition}  
If the analytic extension of the function $g : I \to \mathbb C$, 
\begin{equation}
\label{gauss_map}
g(u):=\begin{cases}
\dfrac{L_1+iL_2}{L_3};\; \text{ if } \gamma' \text{ vanishes identically}\\
\dfrac{{\gamma_1}'+i{\gamma_2}'}{{\gamma_3}'};\; \text { if }  L \text{ vanishes identically}
\end{cases} 
\end{equation}
satisfies $|g(z)| \not \equiv 1$
on some simply connected domain $\Omega \subset \mathbb C$, where $z = u + iv \in \Omega$ and $I\subset \Omega$. Then there is a unique generalized maximal immersion  $X:\Omega\to \mathbb E_1^3$ is given by (for $u_0\in I$ fixed),
$
X(z)= \gamma(u_0)+ {Re}\left(\int_{u_0}^z (\gamma^{\prime}(w)-i L(w)) dw \right)
$
such that $X(u,0) = \gamma(u)$ and $X_v(u,0) = L(u)$. Moreover it has singularity set at least $I$.  After a translation we can assume that $\gamma(u_0)=0$, so we consider the solution as 
\begin{equation}
\label{bjorling_solution}
X_{\gamma,L}(z)= {Re}\left(\int_{u_0}^z (\gamma^{\prime}(w)-i L(w)) dw \right).
\end{equation}

The way singular Bj\"{o}rling data is taken, the singularity set contains an interval $I$ and for all $u$,  $\gamma^{\prime}(u)$ and $L(u)$ do not vanish simultaneously.   Therefore $X_{\gamma, L}$ turns out to be  a maxface in a neighborhood of singular points. In fact the Weierstrass data for the maxface as in the equation  \ref{bjorling_solution} is given by the analytic extension of $f(u) = (\gamma_1^\prime-iL_1)-i(\gamma_2^\prime-iL_2)$ and $g$ as in the equation \ref{gauss_map}. 
\section{Necessary and sufficient conditions on the singular Björling data for prescribed type of singularity }

 In  this section, we will calculate $\alpha, \beta$, and $\eta$ as in the definition \ref{alpha_beta_gamma} for the maxface $X_{\gamma, L}$ at some singularity $t_0\in I$.  We get necessary and suﬃcient conditions on  $\{\gamma,L\}$ so that $X_{\gamma,L}$ have a cuspidal-edge, swallowtails, cuspidal cross caps, cuspidal butterﬂies and cuspidal $S_1^-$ singularities.
\subsection{For cuspidal-edge at $u\in I$.} 
Let $\{\gamma,L\}$ be the singular Björling data as in the deﬁnition \ref{bjorling_data}. With the Gauss map as in the equation \ref{gauss_map}, we calculate $\alpha$ as in the definition \ref{alpha_beta_gamma}. 

At $u\in I,$ if $\gamma^{\prime}(u) \neq 0$, then there exist a real number $c$ such that $L(u) = c\gamma^{\prime}(u)$. So that $f(u) = (1-ic)(\gamma_1^{\prime}(u)-i\gamma_2^{\prime}(u))$ and $g(u) =\dfrac{\gamma_1^{\prime}(u)+i\gamma_2^{\prime}(u)}{\gamma_3^{\prime}(u)}$. In this case, we have $\alpha(u) = \dfrac{g^{\prime}}{g^2f}(u)=\dfrac{g^{\prime}}{g(1-ic)\gamma_3^{\prime}}(u)$. Here replacing the value of f and g we get
$$
\alpha(u)=\frac{\gamma_3^\prime(\gamma_1^{\prime\prime}+i\gamma_2^{\prime\prime})-(\gamma_1^\prime+i\gamma_2^\prime)\gamma_3^{\prime\prime}}{\gamma_3^{\prime 2}(1-ic)\gamma_3^\prime}\frac{\gamma_3^\prime}{\gamma_1^\prime+i\gamma_2^\prime}=-\frac{\gamma_3^{\prime\prime}}{\gamma_3^{\prime 2}(1-ic)}+\frac{(\gamma_1^{\prime\prime}+i\gamma_2^{\prime\prime})(\gamma_1^\prime-i\gamma_2^\prime)}{\gamma_3^\prime(1-ic)(\gamma_1^{\prime 2}+\gamma_2^{\prime 2})}(u).
$$
That is we have, 
$$
\alpha(u)=\frac{1}{\gamma_3^{\prime 3}(1-ic)}[-\gamma_3^{\prime\prime}\gamma_3^\prime+\gamma_1^{\prime\prime}\gamma_1^\prime+\gamma_2^{\prime\prime}\gamma_2^\prime+i(\gamma_1^\prime\gamma_2^{\prime\prime}-\gamma_1^{\prime\prime}\gamma_2^\prime)](u).
$$
We denote:
\begin{equation}
    D(\gamma_{12}^\prime, \gamma_{12}^{\prime\prime}):=\gamma_1^\prime\gamma_2^{\prime\prime}-\gamma_1^{\prime\prime}\gamma_2^\prime;\quad D(L_{12}, L_{12}^\prime):=L_1L_2^\prime-L_2L_1^\prime.
\end{equation}
So that $\alpha(u)=i\dfrac{D(\gamma_{12}^\prime, \gamma_{12}^{\prime\prime})}{\gamma_3^{\prime 3}(1-ic)}$. Moreover $\gamma$ is a null curve and $c=\dfrac{L_3}{\gamma_3^\prime}$, therefore at $u$, when $\gamma^\prime(u)\neq 0$, we get:
\begin{equation}
    \alpha(u)=-\frac{L_3D(\gamma_{12}^\prime, \gamma_{12}^{\prime\prime})}{(\gamma_3^{\prime 2}+L_3^2)\gamma_3^{\prime 2}}+i\frac{D(\gamma_{12}^\prime, \gamma_{12}^{\prime\prime})}{(\gamma_3^{\prime 2}+L_3^2)\gamma_3^{\prime }}.
\end{equation}
Similarly,  for the case when $L(u)\neq 0$, we get the following:
\begin{equation}
\label{cuspidal_edge}
    \alpha(u)=-\frac{D(L_{12}, L_{12}^\prime)}{(\gamma_3^{\prime 2}+L_3^2)L_3}+i\frac{\gamma_3^\prime D(L_{12}, L_{12}^\prime)}{(\gamma_3^{\prime 2}+L_3^2)L_3^2}.
\end{equation}
We know, $u$ is a cuspidal-edge for the maxface if and only if $Re(\alpha)$ and $Im(\alpha)$ at $u$ are non zero. 

Therefore at those points $u\in I$, where $\gamma^\prime\neq 0$ and $u$ is cuspidal-edge for the maxface (as in the equation \ref{bjorling_solution}), we must have $\gamma^\prime\neq 0,L\neq 0$ and $D(\gamma_{12}^\prime,\gamma_{12}^{\prime\prime})\neq 0$ at $u$ and this implies $D(L_{12},L_{12}^\prime) \neq 0$ at $u$. 

On the other hand, at those points $u \in I$, where $L\neq 0$ and $u$ is a cuspidal-edge, we must have $\gamma^\prime\neq 0, L\neq 0$ and $D(L_{12},L_{12}^\prime) \neq 0$ at $u$, and this implies $D(\gamma_{12}^\prime,\gamma_{12}^{\prime\prime}) \neq 0$ at $u$. 

This proves the following;
\begin{prop}
\label{proposition_cuspidaledge}
Let $\{\gamma,L\}$ be the singular Bj\"{o}rling data.  Then the maxface $X_{\gamma,L}$ as in the equation \ref{bjorling_solution} has cuspidal-edge at $u\in I$, if and only if at $u$, $\gamma^\prime\neq 0, L\neq 0$ and $D(\gamma_{12}^\prime,\gamma_{12}^{\prime\prime}) \neq 0$ or $D(L_{12},L_{12}^\prime) \neq 0$. 
\end{prop}
This proposition has many applications, in particular, we will use it to prove the theorem \ref{main_theorem}. Moreover, constructing examples having cuspidal-edge singularity turns out be handy.
\begin{example}
Let $\gamma(u)=(sin\,u,- cos\,u, u)$ and $L(u)=u(cos\,u, sin\,u, 1)$ on $I=(0,1)$. Then we have  $L(u)=u\gamma^\prime(u)$ and $D(\gamma_{12}^\prime,\gamma_{12}^{\prime\prime})=1$.

It is clear that $\gamma^\prime\neq 0,\; L\neq0$ and $D(\gamma_{12}^\prime,\gamma_{12}^{\prime\prime})\neq 0$ on $(0,1)$. So that all are cuspidal-edge on $(0,1)$.
\end{example}
\begin{example}
Let $\gamma(u)=(u-\frac{u^3}{3}, u^2, u+\frac{u^3}{3})$ and $L(u)=u^2(1-u^2,2u, 1+u^2)$ on $I=(0,1)$. Then $L(u)=u^2\gamma^\prime(u)$ and $D(\gamma_{12}^\prime,\gamma_{12}^{\prime\prime})=2(u^2+1)$. It is clear that $\gamma^\prime\neq 0, L\neq0$ and $D(\gamma_{12}^\prime,\gamma_{12}^{\prime\prime})\neq 0$ on $(0,1)$. Hence all points are cuspidal-edge on $(0,1)$.
\end{example}

In the following, we will ﬁnd necessary and suﬃcient conditions on the singular Bj\"{o}rling data such that the  maxface (as in the equation \ref{bjorling_solution}) has swallowtails, cuspidal cross-caps etc, at $u \in I$. 
\subsection{For Swallowtails and cuspidal butterﬂies at $u$.}
If $L\neq 0$ at $u$, then $\gamma^\prime= dL$, where $d$ is a function in a neighborhood of $u$. In this case, from the equation \ref{cuspidal_edge}, we have $\alpha=i\dfrac{D(L_{12}, L_{12}^\prime)}{(d-i)L_3^3}$ therefore $\alpha^\prime=i\frac{D(L_{12}, L_{12}^{\prime\prime})(d-i)L_3^3-D(L_{12}, L_{12}^\prime)(d^\prime L_3^3+3(d-i)L_3^2L_3^\prime)}{(d-i)^2L_3^6}$.

This gives
\begin{equation}
    \beta =\frac{\alpha^\prime}{(d-i)L_3\alpha}=\frac{D(L_{12}, L_{12}^{\prime\prime})(d-i)L_3^3-D(L_{12}, L_{12}^\prime)(d^\prime L_3^3+3(d-i)L_3^2L_3^\prime)}{(d-i)^2L_3^4D(L_{12}, L_{12}^\prime)},
\end{equation}
\begin{equation} \label{beta^prime}
\begin{split}
    \beta^\prime&=\frac{D(L_{12}, L_{12}^{\prime})(d-i)L_3(D(L_{12}, L_{12}^{\prime\prime})+D(L_{12}, L_{12}^{\prime\prime\prime}))}{(d-i)^2L_3^2D^2(L_{12}, L_{12}^\prime)}\\
    &-\frac{D(L_{12}, L_{12}^{\prime\prime})((d-i)L_3^\prime D(L_{12}, L_{12}^\prime)+(d-i) L_3D(L_{12}, L_{12}^{\prime\prime})+d^\prime L_3D(L_{12}, L_{12}^{\prime}))}{(d-i)^2L_3^2D^2(L_{12}, L_{12}^\prime)}\\
    &-\frac{(d-i)^2L_3^2(d^{\prime\prime}L_3+4d^\prime L_3^\prime+3(d-i)L_3^{\prime\prime})}{(d-i)^4L_3^4}\\
    &+\frac{(d^\prime L_3+3(d-i)L_3^\prime)(2(d-i)d^\prime L_3^2+2(d-i)^2L_3L_3^\prime)}{(d-i)^4L_3^4}.
    \end{split}
\end{equation}
Moreover
\begin{equation} \label{eta}
    \eta=\frac{g}{g^\prime}\beta^\prime=\frac{\beta^\prime}{(d-i)L_3\alpha}=-i\frac{\beta^\prime L_3^2}{D(L_{12}, L_{12}^\prime)}.
\end{equation}
We know  $X_{\gamma,L}$  has swallowtails at $u\in I$ if and only if at $u$, $Re\,\alpha\neq 0, Im\,\alpha=0$ and $Re\,\beta\neq0$.

The first two conditions $Re\,\alpha\neq 0$ and $Im\,\alpha=0$ at $u$ if and only if at $u$, $D(L_{12}, L_{12}^\prime)\neq 0, L\neq 0$ and $\gamma^\prime=0$. Since at $u$, $d=\dfrac{\gamma^\prime}{L}=0$,
$$
\beta =\frac{D(L_{12}, L_{12}^{\prime})d^\prime L_3^3+i(D(L_{12}, L_{12}^{\prime\prime})L_3^3-3D(L_{12},L_{12}^\prime)L_3^2L_3^\prime)}{L_3^4D(L_{12}, L_{12}^\prime)}.
$$
Therefore at $u$, $Re\, \beta \neq 0$ if and only if $d^\prime \neq 0$ at $u$. That is $Re \,\beta \neq0$ if and only if $\gamma^{\prime\prime}\neq 0$ at $u$.

On the other hand,  $X_{\gamma,L}$ as in the equation \ref{bjorling_solution} has a cuspidal butterflies at $u\in I$ if and only if at $u$, $Re\,\alpha\neq0, Im\, \alpha=0, Re\,\beta=0$ and $Im\,\eta\neq 0$.

The first two conditions hold at $u$ if and only if at $u$, $D(L_{12}, L_{12}^\prime)\neq0, L\neq0,\gamma^\prime=0$ and $d^\prime=0$. Since at $u$, $d=0$ and $Re \,\beta =0$ if and only if $d^\prime=0$, we get from the equation \ref{beta^prime} and \ref{eta}
\begin{align*}
    \eta=\frac{D(L_{12}, L_{12}^{\prime})(D(L_{12}^\prime, L_{12}^{\prime\prime})+D(L_{12}, L_{12}^{\prime\prime\prime}))}{D^2(L_{12}^\prime, L_{12}^{\prime})}&+\frac{D(L_{12}^\prime, L_{12}^{\prime\prime})(L_3^\prime D(L_{12}^\prime, L_{12}^{\prime})+L_3 D(L_{12}^\prime, L_{12}^{\prime\prime}))}{D^2(L_{12}^\prime, L_{12}^{\prime})}\\
&-i\frac{L_3(d^{\prime\prime}L_3-3iL_3^{\prime\prime})+6iL_3^{\prime 2}}{L_3D(L_{12}^\prime,L_{12}^\prime)}.
\end{align*}
Therefore  at $u, Im\,\eta\neq 0$ if and only if $d^{\prime\prime}\neq 0$. That implies at $u$, $Im\,\eta \neq0$ if and only if $\gamma^{\prime\prime\prime}\neq0$. So we have the following: 
\begin{prop}
\label{proposition_swallowtails}
$X_{\gamma,L}$ has swallowtails at $u$ if and only if at $u$, $\gamma^\prime=0, \gamma^{\prime\prime}\neq0, L\neq0$, and $D(L_{12},L_{12}^\prime)\neq0$.   On the other hand, $X_{\gamma,L}$ has cuspidal butterflies at $u$ if and only if $\gamma^\prime=0, \gamma^{\prime\prime}=0,\gamma^{\prime\prime\prime}\neq0, L\neq0$, and $D(L_{12},L_{12}^\prime)\neq0$.
\end{prop}
Similar calculation gives the following.
\begin{prop}
\label{propsition_crosscaps}
$X_{\gamma,L}$ has cuspidal cross caps at $u$ if and only if at $u$, $\gamma^\prime\neq0, L=0, L^\prime\neq 0$ and $D(\gamma_{12}^\prime,\gamma_{12}^{\prime\prime})\neq0$. And $X_{\gamma,L}$ has cuspidal $S_1^-$ at $u$ if and only if at $u$, $\gamma^\prime\neq0, L=0, L^\prime=0, L^{\prime\prime}\neq0$ and $D(\gamma_{12}^\prime,\gamma_{12}^{\prime\prime})\neq0$.
\end{prop}
In the table \ref{intro:tab}, we summarized all conditions of propositions \ref{proposition_cuspidaledge}, \ref{proposition_swallowtails} and \ref{propsition_crosscaps}. Using the conditions, it is direct to ﬁnd a maxface with singularities: cuspidal-edge, cuspidal cross caps and cuspidal $S_1^-$ as in the following:
\begin{example}
Let $\delta$ be a null real analytic curve and $\mu$ be a null vector ﬁeld  deﬁned on the interval $I$ such that $\mu=\delta^\prime$, $D(\delta_{12}^\prime,\delta_{12}^{\prime\prime})\neq 0$, and both $\delta$ and $\mu$ are never zero on $I$. Let $a,b$ and $c$ be three diﬀerent real numbers on $I$. Now we construct $\gamma(u) = \delta(u)$ and $L(u) = (u-b)(u-c)^2\mu(u)$. Then $\gamma $ and $L$ are Björling data for the maxface $X_{\gamma,L}$ such that $L(u) = (u-b)(u-c)^2\gamma^\prime(u)$. We see that 
\begin{itemize}
    \item at $a$, $L\neq 0, \gamma^\prime\neq 0$ and $D(\gamma_{12}^\prime, \gamma_{12}^{\prime\prime})\neq 0,$
    \item at $b$, $L=0, L^\prime\neq0, \gamma^{\prime}\neq 0$ and $D(\gamma_{12}^\prime, \gamma_{12}^{\prime\prime})\neq 0$ and
    \item at $c$, $ L=0,L^\prime=0, L^{\prime\prime}\neq0, \gamma^\prime\neq 0$ and $D(\gamma_{12}^\prime, \gamma_{12}^{\prime\prime})\neq 0$.
\end{itemize}
Therefore $a,b$ and $c$ are cuspidal-edge, cuspidal crosscaps and cuspidal $S_1^-$ resp. for the maxface $X_{\gamma,L}$. 

Similarly for three diﬀerent real numbers $m,n$ and $p$ if we take $\gamma^\prime(u) = (u -n )(u-p)^2\delta^\prime(u)$ and $L(u) = \mu(u)$, where $\delta$ and $\mu$ are same as above, then $m,n$ and $p$ are cuspidal-edge, swallowtails and cuspidal butterﬂies resp.
\end{example}
\begin{example}
Let $\gamma(u) = (sin\,u,-cos\,u,u)$ and $L(u) = u(u-1)^2(cos\,u,sin\,u,1)$ be the Bj\"{o}rling data then $-1,0$ and $1$ are cuspidal-edge, cuspidal cross caps and cuspidal $S_1^-$ resp. 
\end{example}

We can construct many such examples.  Moreover as we mentioned in the introduction,  other direct application of the table \ref{intro:tab} is to find a sequence of maxfaces converging to other types of singularities.  We discuss this in the next section.

\section{approximating various singularity by a cuspidal-edge }
We start with an example, that will explain the essence of the main  theorem \ref{main_theorem} of this section. 

Let $X_{\gamma, L}$ be the maxface with the singular Bj\"{o}rling data given by $$\gamma(t)= (0,0,0), \;\;L(t)= (1-t^2, 2t, 1+t^2).$$  The maxface $X_{\gamma, L}$ has shrinking singularity on $I= (-1,1)$ (in fact on $\mathbb R)$.  Moreover this is not a front, but below we will give a sequence of maxface (front) converging to $X_{\gamma, L}$ and having cuspidal edge. 

 For $n>1$, we  define 
\begin{eqnarray*}
&& L_n(t)=\left(1- \frac{1}{n}\right)\left(1-t^2, 2t, 1+t^2\right),\\
&&\gamma_n^\prime(t)=\frac{1}{n}L_n(t).
\end{eqnarray*}

\begin{figure}[h]
\centering
\includegraphics[scale=0.4]{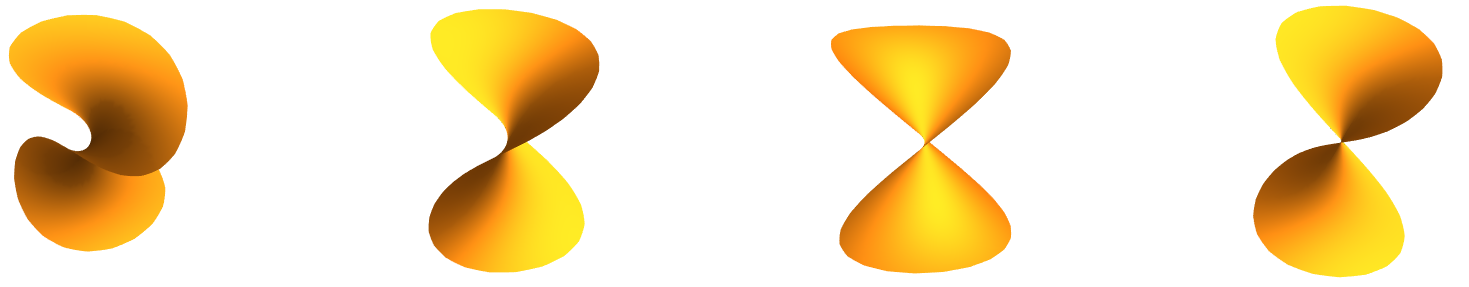}
\caption{Sequence of maxfaces having cuspidal edges that bends to shrinking singularity}
\label{fig:introduction}
\end{figure}
\color{black}
Then for each $n$, the data $\{\gamma_n, L_n\}$ turns out to be a singular Bj\"{o}rling data. Let $X_{\gamma_n, L_n}$ be the corresponding maxface. 
Moreover for each $t\in\mathbb R$, we have $\gamma_n^\prime(t)\neq 0$,  $L_{n}(t)\neq 0$, and  $D(\gamma_{n12}^\prime, \gamma_{n12}^{\prime\prime})\neq 0$. Therefore every point on $\mathbb R$ is a cuspidal edge.   In the figure \ref{fig:introduction}, we have shown the maxfaces $X_{\gamma_n, L_n}$ for $n=3, 5, 15,$ and $50$.\color{black}

In this example we start with a maxface having shrinking singularity and we give a sequence of maxfaces having cuspidal edge ``converging" to the shrinking.  Below we will give general discussion towards this.  First we will define the norm in which we talk about the convergence.

Let $\Omega\subset \mathbb C$ be a bounded simply connected domain, $\overline{\Omega}$ be its closure. Let $X\in C(\overline{\Omega},\mathbb R^3)$, the space of continuous maps. For each $z \in\overline{\Omega}$, we denote
$$
\|X(z)\|:={\rm max}\{X_1(z),X_2(z),X_3(z)\}\,{ \rm and  } \, \|X\|_{\Omega}:=\sup_{z\in{\overline \Omega}}\|X(z)\|.
$$
Here $C(\overline{\Omega},\mathbb R^3)$ becomes a Banach space under the norm $\|.\|_{\Omega}$. 

In the proposition below, we will give a sequence of maxfaces for general $\{\gamma,L\}$.  
\begin{prop}
\label{main_proposition}
Let $X_{\gamma,L}$ be a maxface and for  $t_0 \in I,\gamma^\prime\neq 0,D(\gamma_{12}^\prime, \gamma_{12}^{\prime\prime}) \neq 0$. Then there is a sequence of maxfaces $X_n$ defined in a  neighborhood $\Omega$ of $t_0$ such that each maxface $X_n$  has a cuspidal-edge at $t_0$ and $X_n\to X_{\gamma, L}$  in the norm $\|.\|_\Omega$. 
\end{prop}
\begin{proof}
As $\gamma^\prime(t_0) \neq 0$,  there is an interval $I_1$ containing $t_0$ such that  for all $t \in I_1$, $\gamma_3^\prime(t) \neq 0$.  Without loss of generality we can assume for all $t\in I_1$, $\gamma_3^\prime(t)>0$. \color{black}    On $I_1$, we deﬁne $c(t) = \dfrac{L_3(t)}{\gamma_3^\prime(t)}$. 

For each $n$, we define $\delta_n$ and $\mu_n$ such that 
\begin{eqnarray*}
&&\delta_n^\prime= \gamma^\prime+\left(\frac{1}{n}, \frac{1}{n}, h_n\right); \\
&&\mu_n=\left(c(t)+\frac{1}{n}\right)\delta_n^\prime.
\end{eqnarray*}
Here $h_n= -\gamma_3^\prime+\sqrt{{\gamma_3^\prime}^2+\color{black}2\left(\frac{1}{n^2}+\frac{\gamma_1^\prime+\gamma_2^\prime}{n}\right)}$.

There is a $N$ and $I_2\subset I_1$ containing $t_0$ such that for all $t\in I_2$ and $n>N$,  ${\gamma_3^\prime}^2+\color{black}2\left(\frac{1}{n^2}+\frac{\gamma_1^\prime+\gamma_2^\prime}{n}\right)\neq 0$, and $\frac{1}{n}\neq -c(t_0)$. 

For $n>N$, $\{\delta_n, \mu_n\}$ turn out to be a singular Bj\"{o}rling data  on $I_2$.  These can be extended analytically on some domain $\mathcal{U}$ that contains $I_2$.  We take $\Omega$ a bounded simply connected domain containing $t_0$, such that $\overline{\Omega} \subset\mathcal{U}$.

Moreover we see that at $t_0$ and $n>N_1>N \color{black}$,  $\delta_n^\prime\neq 0$, $\mu_n\neq 0$, and  $$D(\delta_{n12}^\prime,\delta_{n12}^{\prime\prime})\neq 0.
$$ 

Therefore  for $n>N_1\color{black}$, the maxfaces $X_{\delta_n, \mu_n}$  for the singular Bj\"{o}rling data $\{\delta_n, \mu_n\}$,  has cuspidal-edge at $t_0$ and hence cuspidal edge in an interval containing $t_0$.
 
Let $z\in \Omega,$ we have $X_{\delta_n, \mu_n}(z) -X_{\gamma, L}(z)=Re\int_{u_0}^{z}(\delta_n^\prime(w)-\gamma^\prime(w))(1-i(c(w)+\frac{1}{n}))\,dw.$   It is direct to see that 
$\|X_{\delta_n, \mu_n} -X_{\gamma, L}\|_\Omega\to 0$.

\end{proof}

\begin{remark}\label{remark1}
If we have a maxface $X_{\gamma, L}$ with $L\neq0$ and $ D(L_{12}, L_{12}^\prime)\neq0$ at $0$, then with a little change we have a sequence of functions $g_n$ (similar to $h_n$ in the above proposition) and we can take 
$$
\mu_n= L+\left(\frac{1}{n}, \frac{1}{n}, g_n\right)  \mbox{  and  }
\delta_n^\prime=\left(d(t)+\frac{1}{n}\right)\mu_n.
$$
With similar argument as above  we  find a sequence of maxfaces $X_{\delta_n, \mu_n}$ with singular Bj\"{o}rling data $\{\delta_n, \mu_n\}$ having cuspidal-edge at $0$  and $X_{\delta_n, \mu_n}\to X_{\gamma, L}$ in the norm $\|.\|_{\Omega}.$
\end{remark}

\begin{remark}\label{remark2}
For a constant null curve $\gamma$ and a null vector field  $L$ such that for all $t$,  $D(L_{12}, L_{12}^\prime)\neq 0$. The maxface $X_{\gamma,L}$ has shrinking singularity.   For this case we can choose $\gamma_n$ and $L_n$ similar to the example at the beginning of the section.  Moreover little variation will hold for the folded singularity as well. 
\end{remark}
We conclude the article with the following theorem which is a direct consequence of the  above proposition \ref{main_proposition} and remarks \ref{remark1}, \ref{remark2}. 

\begin{theorem}
\label{main_theorem}
Let $X_{\gamma,L}$ be the maxface with singular Bj\"{o}rling data $\{\gamma,L\}$ such that at $t_0 \in I$, it has  shrinking or folded singularity or any of the singular point as in table \ref{intro:tab}.    Then  there is a sequence of maxfaces $X_n$ defined on a domain $\Omega$ containing $t_0$, such that  each $X_n$ has cuspidal-edge at $t_0$. Moreover $X_n\to X_{\gamma, L}$ in norm $\|.\|_\Omega$. 
\end{theorem}
\section{acknowledgement}
Authors are very thankful to the anonymous referees for their valuable comments which helped a lot to improve the article. 
\medskip

\bibliography{main}
\end{document}